\documentclass[11pt, reqno]{amsart}
\usepackage[top=2cm, bottom=2cm, left=2cm, right=2cm]{geometry}

\usepackage{amssymb, amsmath, amsthm, amstext, amscd}
\usepackage{bm, bbm, comment}
\usepackage{mathrsfs}
\usepackage{latexsym, wasysym}
\usepackage{graphics}
\usepackage{verbatim}
\usepackage{amsgen, amsfonts, euscript, enumerate, enumitem, url, calc}
\usepackage{microtype} 
\usepackage{colonequals} 
\usepackage{tikz}
\usetikzlibrary{cd}
\usetikzlibrary{decorations.pathmorphing}
\usepackage{stmaryrd}
\usepackage[all]{xy}
\usepackage{color}
\usepackage[colorlinks, pagebackref]{hyperref}
\hypersetup{
colorlinks=true,
citecolor=blue,
linkcolor=blue,
urlcolor=blue}

\usepackage{url} 
\usepackage{lineno}

\makeatletter
\def\Url@twoslashes{\mathchar`\/\@ifnextchar/{\kern-.2em}{}}
\g@addto@macro\UrlSpecials{\do\/{\Url@twoslashes}}
\makeatother

\newtheorem{theorem}{Theorem}[section]
\newtheorem{lemma}[theorem]{Lemma}
\newtheorem{prop}[theorem]{Proposition}
\newtheorem{corollary}[theorem]{Corollary}

\newtheorem{sub-lemma}[theorem]{Sub-lemma}

\theoremstyle{definition}

\newtheorem{remark}[theorem]{Remark}

\theoremstyle{theorem}
\newtheorem{a-theorem}{A-Theorem}[section]
\newtheorem{a-lemma}[a-theorem]{A-Lemma}
\newtheorem{a-prop}[a-theorem]{A-Proposition}
\newtheorem{a-corollary}{A-Corollary}[section]
\theoremstyle{definition}
\newtheorem{a-definition}[a-theorem]{A-Definition}
\newtheorem{a-example}[a-theorem]{A-Example}
\newtheorem{a-examples}[a-theorem]{A-Examples}
\newtheorem{a-remark}[a-theorem]{A-Remark}

\theoremstyle{theorem}
\newtheorem*{theorem*}{Theorem}
\newtheorem*{question*}{Question}

\newcommand{\suchthat}{\;\ifnum\currentgrouptype=16 \middle\fi|\;}
\makeatletter
\renewcommand\paragraph{\@startsection{paragraph}{4}{\z@}{-3.25ex\@plus -1ex \@minus -.2ex}{1.5ex \@plus .2ex}{\normalfont\normalsize\bfseries}} 

\makeatletter

\renewcommand\subsubsection{\@secnumfont}{\bfseries}%
\renewcommand\subsubsection{\@startsection{subsubsection}{3}
  \z@{.5\linespacing\@plus.7\linespacing}{-.5em}%
  {\normalfont\bfseries}}
  
  \makeatother

\DeclareMathOperator*{\Spec}{\mathrm{Spec}}
\DeclareMathOperator*{\im}{\mathrm{im}}
\DeclareMathOperator*{\rk}{\mathrm{rk}}
\DeclareMathOperator*{\Lie}{\mathrm{Lie}}
\DeclareMathOperator{\vol}{\mathrm{vol}}
\DeclareMathOperator{\diag}{\mathrm{diag}}

\begin{document}

\title{Tamagawa numbers of quasi-split groups over function fields}
\date{\today}
\author{Ralf K\"ohl}
\address{Kiel University, Department of Mathematics, Heinrich-Hecht-Platz 6, 24118 Kiel, Germany}
\email{koehl@math.uni-kiel.de}

\author{M. M. Radhika}
\address{Kerala School of Mathematics, Kozhikode 673571, Kerala, India}
\email{mmr@ksom.res.in}

\author{Ankit Rai}
\address{Chennai Mathematical Institute, Chennai 603103, India}
\email{ankitr@cmi.ac.in}

\begin{abstract}
We use Morris' theory of Eisenstein series for reductive groups over global function fields in order to extend Harder's computation of Tamagawa numbers to quasi-split groups.
\end{abstract}

\maketitle

\tableofcontents

\section{Introduction}
Let $F$ be a global field and $\mathbbmss A$ be the ad\`{e}les over $F$. For an algebraic group $G$ defined over $F$, an invariant $\tau(G) \in \mathbb{R}$ called the Tamagawa number can be associated to $G$.
This is the volume of the space $G(F)\backslash G(\mathbbmss{A})$ with respect to a certain left $G(F)$-invariant Haar measure on $G(\mathbbmss{A})$ called the Tamagawa measure. 
It was conjectured by Weil that for an absolutely simple simply connected algebraic group $G$ over a global field, the Tamagawa number $\tau(G)$ equals $1$.
This was first proved for split groups over number fields by Langlands \cite{Langlands} and over function fields by Harder \cite{Harder-74}.
The proof given by Langlands was rewritten in the adelic language for quasi-split groups by Rapoport \cite{Rapoport-tamagawa} and Lai \cite{Lai-80}, thus giving a unified proof for the split and quasi-split groups over a number field.\\

\noindent
Using Arthur's trace formula, Kottwitz \cite{Kottwitz-tamagawa} proved Weil's conjecture over number fields. The proof of Weil's conjecture over function fields for any semsimple group $G$ was given by Gaitsgory--Lurie \cite{Gaitsgory-Lurie} by a method different than the one used in the earlier works of Langlands, Lai, Rapoport and Kottwitz.
In another direction the theory of Eisenstein series was developed for general reductive groups over function fields in the works of Morris \cite{Morris-82, Morris-82b}. Now that this theory is well developed, it is natural to proceed as in the works of Harder and Lai to directly prove Weil's conjecture for quasi-split groups over function fields. {The present article should be considered as a contribution towards confirming Weil's conjecture for function fields via the strategy used for number fields. The main theorem of this article is as stated.
\begin{theorem} \label{thm:main-theorem}
Let $F$ be a function field of a smooth projective curve over $\mathbb{F}_q$ where $q \neq 2$ and $G$ is a quasi-split semisimple simply connected group over $F$. Then
\[
    \tau(G) = 1.
\]
\end{theorem}
}

\noindent
For non-quasi-split groups either the methods of Kottwitz will have to be used or, alternatively, some other way to establish that the Tamagawa number does not change when passing to inner forms. However, given the unsatisfactory state of the trace formula over function fields, at the moment one cannot proceed further with the methods of Kottwitz. Nevertheless, some progress towards Arthur's trace formula over function fields has been made in \cite{Tuan}. \\

\noindent
Tamagawa \cite{Tamagawa66} originally observed that the group $SO_q(\mathbbmss{A})$ can be endowed with a natural measure such that the Minkowski--Siegel formula is equivalent to the assertion that the Tamagawa number (i.e., the volume with respect to this natural measure) be $2$. Weil \cite{Weil82} subsequently observed that for simply connected groups one should expect the value $1$, which as -- outlined above -- has been confirmed by Kottwitz \cite{Kottwitz-tamagawa} for number fields and by Gaitsgory--Lurie \cite{Gaitsgory-Lurie} for function fields.\\

\noindent
{The organization of the article is as follows. Section \S \ref{subsec:generalities-quas-split-groups} recalls the basics on reductive groups over global and local fields, root systems and sets up the notation for the subsequent sections.
In Section \S \ref{subsec:haar-measures} the Tamagawa measure for semisimple groups, and more generally for reductive groups is defined following the work \cite{Oesterle-84} of Oesterle.
Section \S \ref{subsec:quasi-char} deals with quasi-characters on tori.\\

\noindent
The aim of \S \ref{sec:determining-tamagawa-numbers} is to prove Theorem \ref{thm:main-theorem}. Section \S \ref{subsec:Eisenstein-series} contains generalities on Eisenstein series. In Section \S \ref{subsec:intertwining-operators} we follow the methods of Lai and Rapoport \cite{Lai-80, Rapoport-tamagawa} for computing certain intertwining operators for groups over functions fields and thus, obtain precise information about their poles and zeros (See Theorem \ref{thm:intertwining}). Section \S \ref{subsec:prerequistes-for-the-computation} and \S \ref{subsec:final-computation} is devoted to proving the main theorem.\\

\noindent
The Appendix comprises the proofs of a few technical lemmas used in the main content of this article. These results are well-known and has been added with the hope of improving the exposition of this article.
} \\

\noindent
After preparing the present paper we learned from G. Prasad that our results have also been achieved by E. Kushnirsky in the unpublished part of his PhD thesis \cite{Kushnirsky}.

\section{Basics and Notation}
\label{sec:basics-and-notation}

\subsection{Generalities about quasi-split reductive groups} \label{subsec:generalities-quas-split-groups}

Let $F$ be a function field of a smooth projective curve defined over $\mathbb{F}_q$, $q\neq 2$, of genus $g$. Let $F^{sep}$ be a separable closure of $F$ and $\bar{F}$ be the algebraic closure. For any place $v$ of $F$, let $F_v$ denote the corresponding local field, $k(v)$ be the residue field at $v$, $\mathcal{O}_v$ be the ring of integers in $F_v$, and $\pi_v$ or $\pi_{F_v}$ be a uniformizer of $F_v$.
Let $G$ be a quasi-split group defined over $F$, and $B \subset G$ be a $F$-Borel subgroup fixed throughout this article. Let $B=A\cdot N$ be a Levi decomposition, where $N$ is the unipotent radical and $A$ is a maximal torus defined over $F$. Let $\overline{N}$ be the opposite unipotent radical. Assume that the maximal torus $A$ has been so chosen that the maximal split subtorus $A_d$ of $A$ is the maximal split torus of $G$.\\

\noindent
For any place $v$ of $F$, let $K_v$ denote a special maximal compact subgroup of $G(F_v)$ which always exists by Bruhat--Tits theory. If $G$ is unramified at $v$, then choose $K_v$ to be the hyperspecial maximal compact subgroup.
It is known that a reductive algebraic group $G$ is unramified at almost all places. In other words, for almost all places $v$ of $F$, the group $G \times_F F_v$ admits a smooth reductive model over $\Spec(\mathcal{O}_v)$ and $K_v = G(\mathcal{O}_v)$. Let $S$ denote the set of places of $F$ such that $G$ is unramified outside $S$. \\

\noindent
Let $\mathbbmss{G}, \mathbbmss{B}, \mathbbmss{N}$ and $\mathbbmss{K}$ respectively denote the groups $G(\mathbbmss{A}), B(\mathbbmss{A}), N(\mathbbmss{A})$, and $\underset{v}{\prod} K_v$. We have the Iwasawa decomposition
\[
	\mathbbmss{G} = \mathbbmss{K} \mathbbmss{B}.
\]
\noindent
Recall that a quasi-character is a continuous homomorphism from $A(F)\backslash A(\mathbbmss{A})$ to $\mathbb{C}^{\times}$. A character $\lambda : A \rightarrow \mathbb{G}_m$, defined over $F$, gives a quasi-character $\lambda : A(F)\backslash A(\mathbbmss{A}) \rightarrow q^{\mathbb Z}$ defined to be the composite map
\[
	A(F)\backslash A(\mathbbmss{A}) \rightarrow F^{\times}\backslash \mathbbmss{A}^{\times} \rightarrow q^{\mathbb Z}.
\]
Denote with $X^*(A)$ (resp. $X^*(A_d)$) the group of characters of the torus $A$ (resp. $A_d$) defined over the field of definition of $A$ (resp. $A_d$), and with $\Lambda(A)$ the set of quasi-characters of $A$.

\subsubsection{Root systems}
Let $G \supset B \supset A$ be as before. 
Let $\Pi_F \subset X^*(A_d)$ be the subset of non-trivial weights of $A_d$ on $\mathfrak{g}$. Let $X_*(A_d)$ be the set of cocharacters of $A_d$ and $\Pi^{\vee}_F \subset X_*(A_d)$ be the set of coroots. The root data $(\Pi_F, X^*(A_d))$ can be enhanced to the tuple $(X^*(A_d), \Pi_F, X_*(A_d), \Pi^{\vee}_F)$ called the relative root datum. Denote the absolute root datum by $(X^*(A\times F^{sep}), \Pi,\, X_*(A\times F^{sep}), \Pi^{\vee})$.
Let $X^*_+(A \times F^{sep})$ and $X^*_-(A\times F^{sep})$ respectively denote the weight lattice of the universal cover and the root lattice of $G \times F^{sep}$. In the sequel we shall assume $G$ is simply connected unless stated otherwise. Note that this assumption implies $X^*_-(A\times F^{sep}) = X^*_+(A\times F^{sep})$. Let $\Pi^+$ and $\Pi^+_F$ respectively denote the set of positive absolute roots and the set of positive relative roots of $G$ with respect to $B$.
Let $\rho$ be the half sum of positive relative roots counted with multiplicity. We can also define $\rho$ to be the element of $X^*(A_d)$ or $X^*(A)$ given by $a \mapsto \det\left(\mathrm{Ad}(a|_{\Lie(N)})\right)^{1/2}$.\\

\noindent
Let $W_F \colonequals N_G(A)(F)/Z_G(A)(F)$ be the relative Weyl group and $W = N_G(A)(F^{sep})/Z_G(A)(F^{sep})$ be the absolute Weyl group. We have an embedding $W_F \hookrightarrow W$. 
Recall that there is a $W_F$-equivariant positive definite bilinear form $\langle\cdot, \cdot \rangle : X^*(A_d)_{\mathbb{R}} \times X^*(A_d)_{\mathbb{R}} \rightarrow \mathbb{R}$ such that, the coroot $a^{\vee}$ corresponding to the root $a \in \Pi_F$ is the element $2a/\langle a, a \rangle$ under the isomorphism $X^*(A_d) \simeq X_*(A_d)$ given by $\langle \cdot, \cdot \rangle$.
The set $(\mathbb{Z}\Pi^{\vee}_F)^* \subset X^*(A_d)_{\mathbb{Q}}$, defined under the pairing $X^*(A_d)_{\mathbb{Q}} \times X_*(A_d)_{\mathbb{Q}} \rightarrow \mathbb{Q}$, is called the relative weight lattice of $G$.


\subsubsection{Groups over local fields}
Given a place $v$ of $F$, the group $G \times_F F_v$ is quasi-split as $G$ was assumed to be quasi-split. Furthermore, if $G \times_F F_v$ splits over an unramified extension $E$ of $F_v$, then $G$ admits a smooth reductive model over $\mathcal{O}_v$ and thus a canonical choice of maximal hyperspecial compact subgroup $K_v=G(\mathcal{O}_v)$.
If $G\times_F F_v$ does not split over an unramified extension then it is possible to construct a \textit{parahoric} (see footnote \footnote{Not a reductive group scheme.}) group scheme over $\Spec(\mathcal{O}_v)$ and we define $K_v \colonequals G(\Spec(\mathcal{O}_v))$. This is again a maximal compact subgroup of $G(F_v)$.
We assume these choices have been made and fixed for the rest of the article.\\

\noindent
In the later sections we will need a classification of quasi-split groups over function fields of characteristic $\neq 2$. Thang \cite{Thang-22} gives a complete classification of these groups which was started in the seminal work of Bruhat--Tits \cite{Bruhat-Tits}.
According to the table in \cite{Thang-22}, up to central isogeny there are two quasi-split absolutely simple algebraic groups of relative rank $1$. They are
\begin{enumerate}
	\item $SL_2$
	\item $SU(3, E_v/F_v)$, where $E_v/F_v$ is a quadratic extension.
\end{enumerate}

\subsection{Dual groups} \label{subsec:dual-groups}
We will recall the definition of the dual groups and setup a few more notation here. Let $G$ be a quasi-split group over any field $F$, $A$ be a maximal torus in $G$ defined over $F$, and let $E/F$ be a separable extension such that $G\times_F E$ is a split reductive group.
Let $\Psi(G) \colonequals (X^*(A\times_F E), \Pi_E, X_*(A\times_F E), \Pi_E^{\vee})$ be the root datum of the split reductive group.
Consider the dual root datum $\Psi(G)^{\vee} \colonequals (X_*(A\times_F E), \Pi_E^{\vee}, X^*(A\times_F E), \Pi_E)$ to which is associated a connected semisimple group $\widehat{G}$ over $\mathbb{C}$.
Observe that $\mathrm{Gal}(E/F)$ acts on the root datum $\Psi(G)$ and consequently, we get a Galois action on the dual root datum $\Psi(G)^{\vee}$. This will induce an action of $\mathrm{Gal}(E/F)$ on the associated dual group $\widehat{G}$ as explained below.\\

\noindent
Let $\widehat{A}$ be the maximal torus of $\widehat{G}$. Then the construction of the Langlands dual gives a canonical identification $\eta : \widehat{A}(\mathbb{C}) \rightarrow (X^*(A\times_F E)\otimes \mathbb{C})^{\times}$. Let $\Delta \subset \Pi_E$ be the set of simple roots.
For $\alpha_i \in \Delta$, choose the vectors $X_{\alpha^{\vee}_i} \in \widehat{\mathfrak{g}}\colonequals \Lie(\widehat{G})$ such that for every $\sigma \in \mathrm{Gal}(E/F)$, $\sigma(X_{\alpha^{\vee}_i}) = X_{\sigma\alpha^{\vee}_i}$.
This gives a pinning $(X_*(A\times_F E), \Pi_E^{\vee}, X^*(A\times_F E), \Pi_E, \lbrace X_{\alpha^{\vee}} \rbrace_{\alpha^{\vee}\in \Delta^{\vee}})$ of $\widehat{G}$ equipped with a $\mathrm{Gal}(E/F)$ action.
Since $\mathrm{Gal}(E/F)$ acts on the dual root datum, this action can be lifted to an action on the group $\widehat{G}$ using the splitting of the short exact sequence
\[
    1 \rightarrow \mathrm{Inn}(\widehat{G}) \rightarrow \mathrm{Aut}(\widehat{G}) \rightarrow \mathrm{Aut}(\Psi(G)^\vee) \rightarrow 1
\]
provided by the pinning.

\subsection{Haar measures} \label{subsec:haar-measures}
Let $\omega$ be a left invariant differential form on $G$ of degree $\dim(G)$ defined over $F$. This induces a form $\omega_v$ on $G \times_F F_v$. Denote by $\mathrm{ord}_{e}(\omega_v)$ the number $n$ such that $(\omega_{v})_e(\wedge^{\dim(G)} \Lie(G)) = \pi^n_v$.
The form $\omega_{v}$ defines a left $G(F_v)$-invariant measure on $G(F_v)$ denoted by $\overline{\mu}_{v, \omega}$. For all places $v \notin S$, normalize $\overline{\mu}_{v, \omega}$ as follows
\[
	\overline{\mu}_{v, \omega}(G(\mathcal{O}_v)) = \frac{\sharp G(k(v))}{(\sharp k(v))^{\dim(G)+\mathrm{ord}_e(\omega_v)}}
\]
(cf. \cite[\S 2.5]{Oesterle-84}). For $v\in S$, we refer the reader to \cite[\S 10.1.6]{Bourbaki} for the construction of the Haar measure $\overline{\mu}_{v,\omega}$ (denoted $\mathrm{mod}(\omega_v)$ in Bourbaki) on $G(F_v)$.\\

\noindent
We need more preliminaries before defining the Tamagawa measure on $\mathbbmss G$. For $v \notin S$ denote by $L_v(s, X^*(G))$ the local Artin $L$-function associated to the $\mathrm{Gal}(F^{sep}/F)$-representation $X^*(G \times_F F^{sep})\otimes \mathbb{C}$, where $X^*(G \times_F F^{sep})$ denotes the group of characters of $G$ defined over $F^{sep}$.
Renormalize the measure $\overline{\mu}_{v, \omega}$ on $G(F_v)$ to $L_v(1, X^*(G))\overline{\mu}_{v, \omega}$, and denote the renormalized measure by $\mu_{v, \omega}$. The unnormalized Tamagawa measure on $\mathbbmss{G}$ is then defined to be the measure $\overline{\mu} \colonequals \prod_v \mu_{v, \omega}$. Let $L^S(s,X^*(G))$ denote the product of local $L$-functions outside the set of ramified places $S$. We normalize $\overline{\mu}$ to 
\[
	\mu \colonequals q^{\dim(G)(1-g)} \frac{\overline{\mu}}{\lim_{s\to 1}(1-s)^{\rk{X^*(G)}}L^S(s,X^*(G))},
\]
and call it the Tamagawa measure of $\mathbbmss G$. This  measure is independent of the choice of $\omega$ (cf. \cite[Def. 4.7]{Oesterle-84}). When $G$ is semisimple, $\mathbbmss{G}$ is a unimodular group and hence, the measure $\mu$ descends to $G(F)\backslash \mathbbmss{G}$. The Tamagawa number is then defined as
\[
	\tau(G) \colonequals \vol_{\mu}(G(F)\backslash \mathbbmss{G}).
\]
To extend the definition of the Tamagawa measure for a general reductive group $G$ we proceed as follows. Consider the kernel $\mathbbmss{G}_1$ of the homomorphism $\mathbbmss{G} \xrightarrow{\mathfrak{I}} \hom_{\mathbb{Z}}(X^*(G), q^\mathbb{Z})$ defined by $g\mapsto \big( \chi \mapsto \|\chi(g)\|\big)$ where $g\colonequals(g_v)\in\mathbbmss{G}$.
The image of $\mathbbmss{G}$ under $\mathfrak{I}$ is of finite index (see \cite[\S5.6 Prop.]{Oesterle-84}), and the Tamagawa number of $G$ is defined as
\[
    \tau(G) \colonequals  \frac{\vol_{\mu}(G(F)\backslash \mathbbmss{G}_1)}{(\log q)^{\rk{X^*(G)}}[\hom_{\mathbb{Z}}(X^*(G), q^{\mathbb{Z}}) : \mathfrak{I}(\mathbbmss{G})]}.
\]

\noindent
Choose a Haar measure on $F_v$ such that $\mathrm{vol}(\mathcal{O}_v) = 1$. Let $\overline{da}$ and $\overline{dn}$ be the unnormalized  Tamagawa measures on $A(\mathbbmss{A})$ and $\mathbbmss{N}$ respectively.
Let $dk$ be the unique left invariant (and hence right invariant) Haar measure on $\mathbbmss{K}$ such that $\mathrm{vol}_{dk}(\mathbbmss{K}) = 1$.
Using the Iwasawa decomposition $\mathbbmss{G} = \mathbbmss{N}A(\mathbbmss{A})\mathbbmss{K}$, $\rho^{-2}(a)\overline{dn}\,\overline{da}dk$ is a left invariant Haar measure on $\mathbbmss{G}$. Thus, there exists a positive constant $\kappa$ such that
\[
	\overline{\mu} = \kappa\rho^{-2}(a)\overline{dn}\,\overline{da}dk.
\]
Let $w_0$ be the longest element of the Weyl group that sends all the positive roots to the negative roots and $\dot{w}_0$ be a representative in $N_G(A)(F)$ such that $\dot{w}_{0v}$ belongs to $K_v$ for all $v \notin S$. Then $N(F_v)A(F_v)\dot{w}_0N(F_v)$ is a dense open subset and has full measure. Thus, comparing the measures $\mu_{v,\omega}$ and $\rho^{-2}(a)dn_vda_vdn_v'$ (see footnote \footnote{$dn_v'$ is the measure on $N(F_v)$. The prime is meant to differentiate the two appearances of $N$.}) we get 
\[
	\mu_{v, \omega} = c_v\rho^{-2}(a)dn_vda_vdn'_v,
\]
where $c_v = \frac{L_v(1,X^*(G))}{L_v(1, X^*(A))}$ when $v \notin S$, and $c_v = 1$ otherwise.

\subsection{Quasi-characters on tori} \label{subsec:quasi-char}

Let $A$ be a torus as before and $r=F$-$\rk(G)$. The map 
\begin{align*}
\mathfrak{I} : &A(\mathbbmss{A}) \rightarrow \hom(X^*(A), q^{\mathbb{Z}})\\
               &a \mapsto (\chi \mapsto \|\chi(a)\|)
\end{align*}
defined in \cite{Oesterle-84} induces a map
\[
\begin{tikzcd}[row sep = tiny]
    \mathfrak{I}^*_{\mathbb{C}} : X^*(A) \otimes \mathbb{C} \arrow[r] & \hom(A(\mathbbmss{A})/A(\mathbbmss{A})_1, \mathbb{C}^{\times})\\
    \sum_i c_i\chi_i \arrow[r, mapsto] & (a \mapsto \prod_i \|\chi_i(a)\|^{c_i}).
\end{tikzcd}
\]
The map $\mathfrak{I}^*_{\mathbb{C}}$ is surjective and $X^*(A)\otimes \frac{2\pi\iota}{\log q}\mathbb{Z}$ is a finite index subgroup of $\ker(\mathfrak{I}^*_{\mathbb{C}})$.
Both these assertions follow from the existence of the commutative diagram below
\begin{equation} \label{eqn:comm-diag-A_d-A}
\begin{tikzcd}
    X^*(A) \otimes \mathbb{C} \arrow[d, "\sim" {anchor=north, rotate=90}] \arrow[r, "\mathfrak{I}_{\mathbb{C}}^*"] & \hom(A(\mathbbmss{A})/A(\mathbbmss{A})_1, \mathbb{C}^{\times}) \arrow[d, "\sim" {anchor=south, rotate=90}]\\
    X^*(A_d) \otimes \mathbb{C} \arrow[r, "\mathfrak{I}_{\mathbb{C}}^*"'] & \hom(A_d(\mathbbmss{A})/A_d(\mathbbmss{A})_1, \mathbb{C}^{\times}),
\end{tikzcd}
\end{equation}
where the vertical arrows are induced by the inclusion $A_d \subset A$. The right vertical arrow is an isomorphism since the obvious inclusion $A_d(\mathbbmss{A})/A_d(\mathbbmss{A})_1 \hookrightarrow  A(\mathbbmss{A})/A(\mathbbmss{A})_1$ is an isomorphism. This follows from the fact that the anisotropic part of the torus is contained in $A(\mathbbmss{A})_1$.
The left vertical arrow is an isomorphism since the torus $A$ is quasi-split, which implies that the map $X^*(A) \rightarrow X^*(A_d)$ is injective and the image is of finite index.
Because the kernel of the bottom arrow in \eqref{eqn:comm-diag-A_d-A} is known to be $X^*(A_d)\otimes \frac{2\pi\iota}{\log q}\mathbb{Z}$ (see \cite{Harder-74}), $X^*(A) \otimes  \frac{2\pi\iota}{\log q}\mathbb{Z}$ is a finite index subgroup of the kernel of the top arrow.
The induced map on the quotient is again denoted $\mathfrak{I}^*_\mathbb{C}$
\begin{equation} \label{eqn:iso}
    X^*(A)\otimes \mathbb{C}/\left(X^*(A)\otimes \frac{2\pi\iota}{\log q} \mathbb{Z}\right) \xrightarrow{\mathfrak{I}^*_{\mathbb{C}}} \hom(A(\mathbbmss{A})/A(\mathbbmss{A})_1, \mathbb{C}^{\times}).
\end{equation}

\noindent
Fix a coordinate system on $X^*(A)\otimes \mathbb{C}/\left(X^*(A)\otimes \frac{2\pi\iota}{\log q} \mathbb{Z}\right)$ as follows.
Let $\lbrace \varpi_i \rbrace$ be the fundamental weights of the group $G$. Denote by $[\varpi_i]$ the sum over $\mathrm{Gal}(E/F)$-orbit of $\varpi_i$.
Since $G$ is assumed to be simply connected we have the equality $X^*(A\times F^{sep}) = \oplus_i \mathbb{Z}\varpi_i$. Moreover, $G$ is  quasi-split and hence by Lemma \ref{lemma:quasisplit-tori-structure} $A$ is a quasi-split torus. Now, using \cite[Thm. 2.4]{Oesterle-84} we get that $[\varpi_i]$ is a $\mathbb{Z}$-basis of $X^*(A)$.
The above choice of coordinate system induces the isomorphism
\begin{equation} \label{eqn:coordinate-system}
    \mathbb{Z}^r \xrightarrow{\xi} X^*(A).
\end{equation}
A small computation shows that $\xi(1,1,\dots, 1) = \rho$.
\\

\noindent
For $\lambda \in \Lambda(A)$ define $\Re\lambda(t) \colonequals |\lambda(t)| \in \mathbb{R}$ and   
\[
	\Lambda_{\sigma}(A) \colonequals \lbrace \lambda \in \Lambda(A) \;|\; \Re(\lambda) = \sigma \rbrace.
\]
The latter is a translate of $\Lambda_0(A)$ which is the Pontryagin dual of $A(F)\backslash A(\mathbbmss{A})$. Equip $\Lambda_0(A)$ with the Haar measure $d\lambda$ that is dual to the measure on $A(F)\backslash A(\mathbbmss{A})$ induced by $\overline{da}$.
The measure on $\Lambda_{\sigma}(A)$ is then the unique left $\Lambda_0(A)$-invariant measure such that the volume remains the same. We fix this measure for the future computations.

\subsubsection{Comparison of measures on quasi-characters} \label{subsubsec:constant-c}
The short exact sequence  $$1 \rightarrow A(F)\backslash A(\mathbbmss{A})_1 \rightarrow A(F)\backslash A(\mathbbmss{A}) \rightarrow A(\mathbbmss{A})/A(\mathbbmss{A})_1 \rightarrow 1$$  of locally compact abelian groups gives the exact sequence
\[
1 \rightarrow \hom(A(\mathbbmss{A})/A(\mathbbmss{A})_1, S^1) \rightarrow \Lambda_0(A) \rightarrow \hom(A(F)\backslash A(\mathbbmss{A})_1, S^1) \rightarrow 1.
\]

Since the last term is discrete we get the equality $\hom(A(\mathbbmss{A})/A(\mathbbmss{A})_1, S^1) = \Lambda_0(A)^\circ$.
The pullback of the measure $d\lambda|_{\Lambda_0(A)^\circ}$ along the map $\mathfrak{I}^*_{\mathbb{C}}$, denoted by $d\lambda|_{\frac{X^*(A)\otimes \mathbb{R}}{2\pi /\log(q)X^*(A)}}$, can be compared with the dual measure on $X^*(A) \otimes \mathbb{R}$.
Arguing as in \cite[Lemma 6.7]{Lai-80} we get the following:

\begin{lemma}
\[
    d\lambda|_{\frac{X^*(A)\otimes \mathbb{R}}{2\pi /\log(q)X^*(A)}} = \frac{[\hom(X^*(A), q^{\mathbb{Z}}), \im \mathfrak{I}]} {\mathrm{vol}_{\overline{da}} (A(F)\backslash A(\mathbbmss{A})_1)} \left(\frac{\log q}{2\pi}\right)^rdz_1 \wedge \dots \wedge dz_r.
\]
\end{lemma}
\begin{proof}
Recall the map $\mathbb{C}^r\overset{\xi}{\longrightarrow} X^*(A)\otimes\mathbb{C}$ in \eqref{eqn:coordinate-system} giving the 
isomorphism 
$\mathbb{C}^r/\frac{2\pi}{\log q}\mathbb Z^r\simeq X^*(A)\otimes\mathbb{C}/\frac{2\pi}{\log q}X^*(A).$
Equip the latter space with the measure that assigns mass $1$ to the fundamental domain $X^*(A)\otimes\mathbb{R}/\frac{2\pi}{\log q}X^*(A)$, which under the above isomorphism equals the measure $\left(\frac{\log q}{2\pi}\right)^rdz_1\wedge\cdots\wedge dz_r$.
Denote by $\mathfrak{I}^{\vee} : (\hom(X^*(A), q^{\mathbb{Z}}))^{\vee} \rightarrow \hom(A(\mathbbmss{A})/A(\mathbbmss{A})_1, S^1)$ the map induced by $\mathfrak{I}$ on the Pontryagin dual. We get the following short exact sequence
\[
    1 \rightarrow (\hom(X^*(A), q^{\mathbb{Z}})/\im \mathfrak{I})^{\vee} \rightarrow (\hom(X^*(A), q^{\mathbb{Z}}))^{\vee} \xrightarrow{\mathfrak{I}^{\vee}} \hom(A(\mathbbmss{A})/A(\mathbbmss{A})_1, S^1) \rightarrow 1
\]
The term in the middle is isomorphic to $X^*(A)\otimes \mathbb{R}/\frac{2\pi}{\log q}X^*(A)$ and the first term is abstractly isomorphic to $\hom(X^*(A), q^{\mathbb{Z}})/\im \mathfrak{I}$ since it is finite.
Note that the quotient measure on $A(\mathbbmss{A})/A(\mathbbmss{A})_1$ is $\mathrm{vol}_{\overline{da}}(A(F)\backslash A(\mathbbmss{A})_1)$ times the counting measure and hence, the dual measure $d\lambda$ assigns the mass $1/\mathrm{vol}_{\overline{da}}(A(F)\backslash A(\mathbbmss{A})_1)$ to $\hom(A(\mathbbmss{A})/A(\mathbbmss{A})_1, S^1)$.
The pullback of this measure along $\mathfrak{I}^*_{\mathbb{C}}$ is a Haar measure which assigns mass $\frac{|\hom(X^*(A), q^{\mathbb{Z}})/\im \mathfrak{I}|}{\mathrm{vol}_{\overline{da}}(A(F)\backslash A(\mathbbmss{A})_1)}$ to $(\hom(X^*(A), q^{\mathbb{Z}}))^{\vee}$, whereas the Haar measure $\left(\frac{\log q}{2\pi}\right)^rdz_1\wedge\cdots\wedge dz_r$ assigns it mass 1. Hence the claim.
\end{proof}

\section{Determining the Tamagawa numbers} \label{sec:determining-tamagawa-numbers}

\subsection{Eisenstein series} \label{subsec:Eisenstein-series}
Let $\varphi : \mathbbmss{N}B(F)\backslash \mathbbmss{G}/\mathbbmss{K}\rightarrow\mathbb C$ be a compactly supported measurable function. Let $\varphi_v$ denote the local components of $\varphi$ such that $\varphi = \underset{v}{\prod}{\varphi_v}$. For any $\lambda \in \Lambda(A)$ the Fourier transform is defined as
\[
	\widehat{\varphi}(\lambda)(g) \colonequals \underset{A(F)\backslash A(\mathbbmss{A})}{\int} \varphi(ag)\lambda^{-1}(a)\rho^{-1}(a) \overline{da}.
\]
Let $\widehat{\varphi}(\lambda)_v$ denote the restriction of $\widehat{\varphi}(\lambda)$ to $G(F_v)$. Then for $g=(g_v)_v$ we have $\widehat{\varphi}(\lambda)(g) = \underset{v}{\prod}\widehat{\varphi}(\lambda)_v(g_v)$.
Note that $\widehat{\varphi}(\lambda)(g)$ is determined by its value at 1 and we denote this value simply by $\widehat{\varphi}(\lambda)$. On applying Fourier inversion 
\[
	\varphi(g) = \int_{\Re(\lambda)=\lambda_0} \widehat{\varphi}(\lambda)(g) d\lambda, 
\]
where $\lambda_0$ is such that for any coroot $\alpha^{\vee}$ the composite $F^*\backslash \mathbbmss{A}^* \xrightarrow{\alpha^{\vee}} A(F)\backslash A(\mathbbmss{A}) \xrightarrow{|\lambda|} \mathbb{R}^{\times}$ given by $|\cdot|^{s_{\alpha}}$ satisfies $s_{\alpha}>1$.
For $\varphi$ as above define the theta series
\[
	\theta_{\varphi}(g) \colonequals \sum_{\gamma \in B(F)\backslash G(F)} \varphi(\gamma g).
\]
The above series converges uniformly on compact subsets of $G(F)\backslash \mathbbmss{G}$ (see \cite[\S 2.3]{Morris-82}). In fact, the support of $\theta_{\varphi}$ is compact and hence, it is in $L^2(G(F)\backslash \mathbbmss{G})$. Define
\[
	E(g, \widehat{\varphi}(\lambda)) \colonequals \sum_{\gamma \in B(F)\backslash G(F)} \widehat{\varphi}(\lambda)(\gamma g).
\]
We have 
\begin{align*}
\theta_{\varphi}(g) = \sum_{\gamma \in B(F)\backslash G(F)} \varphi(\gamma g) &= \sum_{\gamma \in B(F)\backslash G(F)} \int_{\Re(\lambda) = \lambda_0} \widehat{\varphi}(\lambda)(\gamma g) d\lambda\\
				&\overset{\ast}{=} \int_{\Re(\lambda) = \lambda_0} \sum_{\gamma \in B(F)\backslash G(F)} \widehat{\varphi}(\lambda)(\gamma g) d\lambda \\
				&=\int_{\Re(\lambda) = \lambda_0} E(g, \widehat{\varphi}(\lambda)) d\lambda.
\end{align*}
The assumption on $\lambda_0$ is used in the equality marked with $\ast$ above. The Eisenstein series $E(g, \widehat{\varphi}(\lambda))$ is a priori defined on the domain $\lambda \in X^*(A)\otimes \mathbb{C}/\left( X^*(A)\otimes \frac{2\pi\iota}{\log q}\mathbb Z \right)$ and $\Re(\lambda) - \rho \in C$ (see footnote \footnote{$C$ is the positive Weyl chamber;
or in other words $(\Re(\lambda), \alpha) > (\alpha, \rho)$.}), but can be continued meromorphically to all of $X^*(A)\otimes \mathbb{C}/\left( X^*(A)\otimes \frac{2\pi\iota}{\log q}\mathbb Z\right)$ (see footnote \footnote{The Eisenstein series can be extended to other connected components as well.}).\\

\noindent
For $w \in W_F$, let $\dot{w}$ denote a lift to $N_G(A)(F)$. Set ${}^{\dot{w}}N = \dot{w}N\dot{w}^{-1}$ and $N^{\dot w}=\dot w \overline{N}\dot{w}^{-1}\cap N$. Recall the definition of the local and global intertwining operators,

\begin{align}
\big(M_v(w, \lambda)\widehat{\varphi}(\lambda)_v\big)(g_v) &= \underset{N^{\dot{w}}(F_v)}{\int}\widehat{\varphi}(\lambda)_v(\dot{w}n_vg_v) dn_v \quad \text{ for any $g_v \in G(F_v)$}, \label{eqn:intertwine-local}\\
\big(M(w, \lambda)\widehat{\varphi}(\lambda)\big)(g) &= \underset{{}^{\dot{w}}N(F)\cap N(F)\backslash \mathbbmss{N}}{\int} \widehat{\varphi}(\lambda)(\dot{w}ng) \overline{dn} \notag \\
&= \mathrm{vol}\big({}^{\dot{w}}N(F)\cap N(F)\backslash ({}^{\dot{w}}\mathbbmss N \cap\mathbbmss N)\big) \underset{\dot{w}\mathbbmss{N}\dot{w}\cap \mathbbmss{N}\backslash \mathbbmss{N}}{\int} \widehat{\varphi}(\lambda)(\dot{w}ng)\overline{dn} \notag \\
&= \mathrm{vol}\big({}^{\dot{w}}N(F)\cap N(F)\backslash ({}^{\dot{w}}\mathbbmss N \cap\mathbbmss N)\big) \underset{\mathbbmss{N}^{\dot{w}}}{\int} \widehat{\varphi}(\lambda)(\dot{w}ng)\overline{dn}. \label{eqn:intertwine-global}
\end{align}
The last equality in \eqref{eqn:intertwine-global} follows from $^{\dot{w}}N(F)\cap N(F) \backslash \mathbbmss N={}^{\dot{w}}N(F)\cap N(F)\backslash \big({}^{\dot{w}}\mathbbmss N \cap\mathbbmss N\big) \mathbbmss N^{\dot{w}}$, and the left $\mathbbmss N$-invariance of $\widehat{\varphi}(\lambda)$.
Observe that for $w=w_0$ the group ${}^{\dot{w}}\mathbbmss{N} \cap \mathbbmss{N}$ is trivial and hence combining equations \eqref{eqn:intertwine-local} and \eqref{eqn:intertwine-global} we get
\[
    M(w_0, \lambda) = \underset{v}{\prod}' M_v(w_0, \lambda).
\]
Note that $M(w, \lambda)(\widehat{\varphi}(\lambda))$ is right $\mathbbmss{K}$-invariant, left $\mathbbmss{N}$-invariant and transforms via $\lambda$ on $A(\mathbbmss{A})$. Hence $M(w, \lambda)(\widehat{\varphi}(\lambda))$ is a scalar multiple of $\widehat{\varphi}(\lambda)$.
The operator $M(w, \lambda)$ can thus be confused with the scalar. We record here the following proposition and a lemma to be used in the later sections.

\begin{prop}[\cite{Morris-82}, \S 3.1, Lemma]
\label{lemma:morris}
The constant term of the Eisenstein series is given by the following formula
\[
	E^{B}(g)\colonequals \int_{N(F)\backslash \mathbbmss{N}} E(ng, \widehat{\varphi}(\lambda)) \overline{dn} = \sum_{w\in W_F} \big(M(w, \lambda)\widehat{\varphi}(\lambda)\big)(g).
\]
\end{prop}

\begin{lemma}
\[
	(\theta_{\varphi}, 1) = \int_{B(F)\backslash \mathbbmss{G}} \varphi(g) \overline{d\mu} = \kappa\int_{N(F)\backslash\mathbbmss N} \int_{A(F)\backslash A(\mathbbmss{A})} \int_K \varphi(nak)\rho^{-2}(a) \overline{dn}\,\overline{da}dk = \kappa q^{-\dim(N)(1-g)} \widehat{\varphi}(\rho).
\]
\end{lemma}

\subsection{Intertwining operators} \label{subsec:intertwining-operators}
The aim of this subsection is to prove the following:
\begin{theorem} \label{thm:intertwining}
The intertwining operator $M(w_0, \lambda)$ has a simple pole along each of the hyperplanes $s_i=1$ in the region $1-\epsilon <\Re(s_i) < 1+\epsilon$ for some $\epsilon>0$. In particular, $M(w_0, s\rho)$ has a pole of order $F$-$\rk(G)$ at $s=1$.
\end{theorem}

\noindent
We reduce the calculation of the integrals defining certain intertwining operators to the case of quasi-split semisimple simply connected rank 1 groups following \cite{Lai-80, Rapoport-tamagawa}.
The strategy used by Lai and Rapoport is shown to work in a similar manner over function fields. As a result we obtain Theorem \ref{thm:intertwining} which is crucial to implement the arguments of Harder in order to prove Theorem \ref{thm:projection-onto-constants}. We remark here that, unlike the strategy followed in the present article, Harder explicitly computes an expression for the Eisenstein series (see \cite[\S 2.3]{Harder-74}) and concludes Theorem \ref{thm:intertwining} as a corollary of his results.

\subsubsection{Local intertwining operators}

We will require the computation of the local intertwining operators $M_v(w_0,\rho)$ for the ramified and unramified places of $F$.

\paragraph{ $M_v(w_0, \rho)$ for ramified places}
\label{sec:intertwining-ramified}
Let $G$, $F$ and $S$ be as in Section \ref{sec:basics-and-notation}.
Thus, for any $v \notin S$ the group $G\times_F F_v$ splits over an unramified extension of $F_v$. Let $\dot{w}_0$ denote a representative in $N_G(A)(F)$ of the longest Weyl group element $w_0 \in W_F$ as in \ref{subsec:haar-measures}. Let $\mathbbmss A_S$ denote the ring of ad\`{e}les over $F$  with trivial component outside $S$. For $n \in N(\mathbbmss{A}_S)$ we write the Iwasawa decomposition of $\dot{w}_0n$ as
\[
	\dot{w}_0n = n_1(n)a(n)k(n) \in N(\mathbbmss{A})A(\mathbbmss{A})\mathbbmss{K}.
\]


\begin{prop}
For any finite set $S'$ containing the set of ramified places, let $M_{S'}(w_0,\rho)=\prod\limits_{v\in S'}M_v(w_0,\rho)$. Then
\[
M_{S'}(w_0, \rho) = \int_{(\dot{w}_0 N)(\mathbbmss{A}_{S'})} |\rho|^2(a(n)) \overline{dn} = \kappa \left(\prod_{v \notin S'}\mathrm{vol}(K_v)\right)^{-1} \left(\prod_{v \in S'}c_v\right).
\]
\end{prop}
\begin{proof}
Let $f$ be a right $\mathbbmss{K}$-invariant function on $\mathbbmss{G}$ defined for any $g = nak \in \mathbbmss{G}$ as follows
\[
f(g) = 
\begin{cases}
0 & \text{if } g_v \notin K_v \text{ for some } v \notin S',\\
h(n_{S'}, a_{S'}) & \text{otherwise, where } n_{S'} = (n_v)_{v \in S'}, a_{S'} = (a_v)_{v \in S'} \\ & \text{and } h:N(\mathbbmss A_{S'})\times A(\mathbbmss A_{S'})\to \mathbb{R} \text{ is any integrable function }
\end{cases}
\]
Using the equality $\overline{\mu} = \kappa |\rho|^{-2}(a)\overline{dn}\,\overline{da}dk$ we get
\begin{align*}
\underset{\mathbbmss{G}}{\int} f(g)\overline{\mu} &= \kappa \underset{\mathbbmss{N}A(\mathbbmss{A})\mathbbmss{K}}{\int} f(nak)|\rho|^{-2}(a) \overline{dn}\,\overline{da}dk\\
                                       &= \kappa \underset{N(\mathbbmss{A}_{S'})A(\mathbbmss{A}_{S'})}{\int} h(n_{S'}, a_{S'}) |\rho|^{-2}(a) \overline{dn}\,\overline{da}.
\end{align*}
The largest Bruhat cell $B\dot{w}_0N$ has full measure with respect to $\overline{\mu}$ and hence the left hand side of the above integral equals the integral on this cell.
Using the Iwasawa decomposition of $\dot{w}_0n'\in \dot{w}_0N$ in the Bruhat decomposition of $g$ we get, $na\dot{w}_0n' = nan_1(n')a(n')k(n')$. In the following we omit the subscript $S'$ in the integrand for convenience. We further let $\kappa'\colonequals\left(\prod_{v \notin S'}\mathrm{vol}(K_v)\right) \left(\prod_{v \in S'}c_v\right)$.
\begin{align}
\underset{\mathbbmss{B}\dot{w}_0\mathbbmss{N}}{\int} f(g)\overline{\mu} &= \kappa' \underset{B(\mathbbmss{A}_{S'})\dot{w}_0N(\mathbbmss{A}_{S'})}{\int} f(nan_1(n')a^{-1}aa(n')k(n')) |\rho|^{-2}(a)\overline{dn}\,\overline{da}\,\overline{dn'} \notag\\
                                                            &= \kappa'\underset{\dot{w}_0N(\mathbbmss{A}_{S'})}{\int} \underset{A(\mathbbmss{A}_{S'})}{\int} \underset{N(\mathbbmss{A}_{S'})}{\int} h(n(an_1(n')a^{-1}), aa(n')) |\rho|^{-2}(a)\overline{dn}\,\overline{da}\,\overline{dn'} \notag \\
                                                            &= \kappa' \underset{\dot{w}_0N(\mathbbmss{A}_{S'})}{\int} \underset{A(\mathbbmss{A}_{S'})}{\int} \underset{N(\mathbbmss{A}_{S'})}{\int} h(n, aa(n')) |\rho|^{-2}(a)\overline{dn}\,\overline{da}\,\overline{dn'} \label{eqn:integral-largest-bruhat}
\end{align}
The last equality follows from the right invariance of the Haar measure $\overline{dn}\,$ on $\mathbbmss{N}$. Now using the right invariance of $\overline{da}$ we obtain 

\begin{align*}
\eqref{eqn:integral-largest-bruhat}   &= \kappa'\underset{\dot{w}_0N(\mathbbmss{A}_{S'})}{\int} \underset{A(\mathbbmss{A}_{S'})}{\int} \underset{N(\mathbbmss{A}_{S'})}{\int} h(n, a)|\rho|^{-2}(a) |\rho|^2(a(n')) \overline{dn}\,\overline{da}\,\overline{dn'}\\
                                                            &= \kappa'\underset{\dot{w}_0N(\mathbbmss{A}_{S'})}{\int} |\rho|^2(a(n')) \overline{dn'} \underset{A(\mathbbmss{A}_{S'})}{\int} \underset{N(\mathbbmss{A}_{S'})}{\int} h(n, a)|\rho|^{-2}(a) \overline{dn}\,\overline{da}.
\end{align*}
Since $h$ can be chosen such that the integral is non zero, comparing the right hand side of the two equations implies
\[
\kappa'\underset{\dot{w}_0N(\mathbbmss{A}_{S'})}{\int} |\rho|^2(a(n')) \overline{dn'} = \kappa.
\]
Thus, we get the desired result
\[
    \underset{\dot{w}_0N(\mathbbmss{A}_{S'})}{\int} |\rho|^2(a(n')) \overline{dn'} = \kappa \left(\prod_{v \notin S'}\mathrm{vol}(K_v)\right)^{-1} \left(\prod_{v \in S'}c_v\right).
\]

\end{proof}

\paragraph{ $M_v(w_0, \rho)$ for unramified places}
\begin{corollary}
Suppose $v \notin S$. 
Then
\[
	M_v(w_0, \rho) = \int_{\dot{w}_{0v}N_v} |\rho|^2a(n_v) dn_v = \mathrm{vol}(K_v)\frac{L_v(1, X^*(A))}{L_v(1,X^*(G))}.
\]
\end{corollary}

\noindent
The corollary simply follows from the proposition above by expanding the set $S$ to $S' = S\cup \lbrace v \rbrace$.
We record here a lemma regarding the holomorphicity of the local intertwining operators and refer the reader to \cite[Lemma 6.9]{Lai-80} for a proof. 
\begin{lemma} \label{lemma:holomorphy-local-intertwiner}
The intertwining operators $M_v(w_0, \lambda)$ are holomorphic for $\lambda \in \Lambda_0(A)$ and $\Re(\lambda) \geq 1$.
\end{lemma}

\paragraph{Local intertwining operators for rank 1 groups}

We continue to denote a non-archimedean local field by $F_v$. Let $G$ be a semisimple group defined over $F_v$ with $F_v$-rank 1 which splits over an unramified extension $E_v$ of $F_v$. Let $\dot{w}_{0v}\in K_v$ be a representative of the longest Weyl group element $w_0 \in W_F$ as in \ref{subsec:haar-measures}. Using the right $K_v$-invariance of $\widehat{\varphi}(\lambda)$, the integral in \eqref{eqn:intertwine-local} is equal to
\[
	\int_{N(F_v)} \widehat{\varphi}(\lambda)(\dot{w}_{0v}n_v\dot{w}_{0v}^{-1})dn_v = \int_{\overline{N}(F_v)} \widehat{\varphi}(\lambda)(n_v) dn_v.
\]
This integral depends only on the choice of the Haar measure on the group $\overline{N}(F_v)$, which is taken to be the Tamagawa measure on $\overline{N}(F_v)$.
Theorem \ref{thm:reduction} will allow us to reduce the computation of the local intertwining operators to the case of $F$-rank one groups. Here we will give an explicit computation of $M_v(w_0, \lambda)$ for $F$-rank one groups.\\


\noindent
Fix $\lambda \in X^*(A) \otimes \mathbb{C}$ and suppose that $\hat{t} \in \widehat{A}$ is such that for any $\mu \in X_*(A\times_{F_v} E_v)$ the equality $\hat{t}(\mu) = |\pi_v|^{\langle\lambda, \mu\rangle}$ holds.
Let $\widehat{\mathfrak{u}}$ be the Lie subalgebra of $\widehat{\mathfrak{g}}$ corresponding to the unipotent radical $N$. 
\begin{theorem} \label{thm:local-inter-rank-1-ur}
Suppose $E_v/F_v$ is the unramified extension that splits $G$ and let $\sigma \in \mathrm{Gal}(E_v/F_v)$ be the Frobenius element.
Then
\begin{equation} \label{eqn:intertwiner-local-comp}
	M_v(w_0, \lambda) = \frac{\det\left(I - |\pi_v|\mathrm{Ad}(\sigma\hat{t})|_{\widehat{\mathfrak u}}\right)}{\det\left(I - \mathrm{Ad}(\sigma\hat{t})|_{\widehat{\mathfrak{u}}}\right)}.
\end{equation}
\end{theorem}
\begin{proof}
The proof will be done in stages by verifying the above formula first for semisimple rank 1 groups and then for higher rank groups.\\

\noindent
\textbf{Step 1.} The theorem is true in the case of absolutely simple simply connected groups of semisimple $F_v$-rank 1. We quote the results of Rapoport and Lai below.

\begin{prop}[\cite{Rapoport-tamagawa}, \S 4.4(a)]
The intertwining operator $M_v(w_0,s\rho)$ for the group $SL_2$ is given by
\[
	M_v(w_0,s\rho) = \frac{(1-q^{-s-1})}{(1-q^{-s})}.
\]
\end{prop}

\begin{prop}[\cite{Lai-80}, Prop. 3.4] \label{prop:local-intertwining-formula-rank-one}
Let $E_v/F_v$ be a quadratic unramified extension of $F_v$ and $SU(3,E_v/F_v)$ be the quasi-split group defined over $F_v$. Suppose $2$ is invertible in $F_v$. Then
\[
	M_v(w_0,s\rho) = \frac{(1-q^{-2s-2})(1+q^{-2s-1})}{(1-q^{-2s})(1+q^{-2s})}.
\]
\end{prop}

\noindent
\textbf{Step 2.} If the theorem is true for $G$ then it is true for any central isogeny $\widetilde{G} \rightarrow G$.\\

\noindent
Let $\widetilde{G} \rightarrow G$ be a central isogeny. The notation $\widetilde{~~}$ will denote the corresponding objects for $\widetilde{G}$. It is clear that the right hand side of \eqref{eqn:intertwiner-local-comp} is the same for $\widetilde{G}$ and $G$.
Further, the isogeny induces the isomorphisms $\widetilde{W}_{F_v} \xrightarrow{\sim} W_{F_v}$ and $X^*(A)\otimes \mathbb{C} \xrightarrow{\sim} X^*(\widetilde{A}) \otimes \mathbb{C}$, where the image of $\rho$ under the latter isomorphism is $\widetilde{\rho}$. Also, the images of 
$\widetilde{N},\ \widetilde{A}$ and $\widetilde{K}$ are $N, \ A $ and $K$ respectively; and  $\overline{\widetilde{N}} \xrightarrow{\sim} \overline{N}$. Thus the image of $\widetilde{n}\widetilde{a}\widetilde{k}$ maps to $nak$ which is the Iwasawa decomposition.\\

\noindent
\textbf{Step 3.} Let $G = \mathrm{Res}_{E'_v/F_v}G'$ for a quasi-split simply connected semisimple group $G'$ defined over $E'_v$ which splits over $E_v$ and let the degree of the unramified extension $E'_v/F_v$ be $n$. If the theorem is true for $G'$ then it is true for $G$.\\

\noindent
Alphabets with superscript $'$ will denote the corresponding objects for the group $G'$. The Weyl groups $W'_{E'_v}$ and $W_{F_v}$ are the same. Also, if $A = R_{E'_v/F_v}(A')$ we can identify $\widehat{A}$ with $\prod_{\mathrm{Gal}(E'_v/F_v)}\widehat{A'}$.
We have $\widehat{\mathfrak{u}} = \prod_{\mathrm{Gal}(E'_v/F_v)}\widehat{\mathfrak{u}'}$. Since $\lambda \in X^*(A)\otimes\mathbb C$ we get that $\widehat{t} \in \widehat{A}$ is mapped to a diagonal element $\diag(\widehat{t'}, \widehat{t'}, \cdots, \widehat{t'}) \in \prod_{\mathrm{Gal}(E'_v/F_v)}\widehat{A'}$ under the identification above.

\indent
\[
I-\mathrm{Ad}(\sigma\hat{t}) = \left( \begin{tabular}{c|c|c|c|c}
&&&&\\
$I$ & & & & $-\mathrm{Ad}(\sigma'\widehat{t'})$\\
\hline
&&&&\\
$-\mathrm{Ad}(\sigma'\widehat{t'})$ & $I$ & & & \\
\hline
&&&&\\
 & $-\mathrm{Ad}(\sigma'\widehat{t'})$ & \phantom{ksc} $I$ \phantom{kfj}&&\\
\hline
&&&&\\
& & $\ddots$ & $\ddots$ &\\
&&&&\\
\hline
&&&&\\
 &&& $-\mathrm{Ad}(\sigma'\widehat{t'})$ & $I$
\end{tabular}
\right)
\]

\noindent
Note that $\mathrm{Ad}(\hat{t})$ is a diagonal matrix. Then
\begin{equation} \label{eqn:determinant_ratio}
\begin{split}
	\det\left(I-\mathrm{Ad}(\sigma\hat{t})\right) &= \det\left( I - (\mathrm{Ad}(\sigma'\widehat{t'}))^n |_{\widehat{\mathfrak{u}'}} \right) = \det\left( I - (\mathrm{Ad}(\sigma'\widehat{t'})^n) |_{\widehat{\mathfrak{u}'}} \right),\\
    \text{and \hspace{0.3cm}  } \det\left(I - |\pi_v|\mathrm{Ad}(\sigma\hat{t})|_{\widehat{\mathfrak u}}\right) &= \det\left( I - |\pi'_v|'\,\mathrm{Ad}(\sigma'\widehat{t'}^n)|_{\widehat{\mathfrak{u}'}} \right),
 \end{split}
\end{equation}
where $|\cdot|'$ is the valuation on $E'_v$. The left hand side of \eqref{eqn:intertwiner-local-comp} for $G$ is equal to 
\[
	\int_{\overline{N}^{w_0}} |\lambda|(a(\overline{n}))\cdot|\rho|(a(\overline{n}))d\overline{n} = \int_{\overline{N'}^{w_0}} |N_{E_v'/F_v}\lambda'|\cdot|N_{E_v'/F_v}\rho'|(a'(\overline{n}')) d\overline{n}' = \int_{\overline{N'}^{w_0}} |\lambda'|'|\rho'|'(a'(\overline{n}')) d\overline{n}' = M'_v(w_0, \lambda'),
\]
where $\lambda' \in X^*(A')\otimes \mathbb{C}$ corresponds to $\lambda \in X^*(A)\otimes \mathbb{C}$ under the identification $X^*(A') \simeq X^*(A)$ via the norm map. Finally, the equality $M_v(w_0,\lambda)=M_{v'}(w_0,\lambda')$ along with \eqref{eqn:determinant_ratio} and Lemma \ref{lemma:dual-tori-eta-map} proves that \eqref{eqn:intertwiner-local-comp} holds for $G$.
\end{proof}

\paragraph{Formula of Bhanu-Murthy, Gindikin and Karpelevitch}
Let $F_v$ be a non archimedean local field, $G$ a semisimple linear algebraic group defined over $F_v$, $P_0$ (can take it to be $B$ in our case) a minimal parabolic over $F_v$, and $N_0$ be its unipotent radical.
Let $A_0$ be the maximal torus contained in $P_0$, defined over $F_v$, and let $K$ be a special maximal compact subgroup of $G$ which satisfies $P_0(F_v)K = G(K_v)$.
Denote by $X^*(A_0)$ the set of characters of $A_0$ defined over $F_v$. Let $G(\alpha) \subset G$ be the semisimple group of rank one whose Lie algebra is generated by the radical space corresponding to the roots that are multiples of $\alpha$ (here $\alpha$ is indivisible).
Note that this is the same as the derived subgroup of the centralizer of $\ker(\alpha)^0_{red}$. It is known that this will be a connected simply connected smooth group scheme of rank 1 (see \cite[Cor. 9.5.11]{Conrad-tony}).
Define $P_0(\alpha) \colonequals P_0 \cap G(\alpha)$, then $N_0(\alpha), A_0(\alpha)$ and $K(\alpha)$ can be defined similarly sharing the same properties as $N_0, A_0$ and $K$, but with respect to the group $G(\alpha)$.
There is a unique opposite parabolic subgroup $P_0(-\alpha)$ which we denote by $\overline{P}_0(\alpha)$. Let $\overline{N}(\alpha)$ be the unipotent radical of the opposite parabolic.\\

\noindent
Suppose $P$ is a parabolic defined over $F_v$ such that $P \supset P_0$, and let $N$ be the unipotent radical of $P$. Let $\Pi'_+(P)$ be the set of indivisible roots $\alpha$ such that $N(\alpha) \subset N$. We denote by $\overline{N}$ the opposite unipotent group to $N$.
For $\lambda \in X^*(A_0) \otimes \mathbb{C}$, denote by $\Phi^{\lambda}$ the function over $G(F_v)$ that associates to the element $g_v = n_v.a_v.k_v \in G(F_v)$ the complex number $\Phi^{\lambda}(g_v) = |\lambda|(a_v)|\rho|(a_v)$.
Let $\lambda(\alpha)$ denote the projection of $\lambda \in X^*(A_0) \otimes \mathbb{C}$ under the map $X^*(A_0) \otimes \mathbb{C} \rightarrow X^*(A_0(\alpha)) \otimes \mathbb{C}$ induced by the natural inclusion $A_0(\alpha) \subset A_0$.
We state the following:
\begin{theorem}
\label{thm:reduction}
The integral
\[
	\underset{\overline{N}(F_v)}{\int} \Phi^{\lambda}(\bar{n}_v) d\bar{n}_v
\]
converges for any $\lambda \in X^*(A_0) \otimes \mathbb{C}$ with $\mathrm{Re}(\langle \lambda, \alpha^{\vee} \rangle) > 0$ for all $\alpha \in \Pi'_+(P)$. There exists a constant depending on the choice of Haar measure and up to this constant the value is
\[
	\prod_{\alpha \in \Pi'_+(P)} \underset{\overline{N}(\alpha)(F_v)}{\int} \Phi^{\lambda(\alpha)}(\bar{n}_v) d\bar{n}_v.
\]
If the semisimple group $G$ is the local place of a semisimple group defined over a global field and if the Haar measure is deduced from the Tamagawa measure then this constant is 1 for almost all places $v$.
\end{theorem}

\noindent
As an application of the above we have a straightforward generalization of Proposition \ref{prop:local-intertwining-formula-rank-one}. We refer the interested readers to \cite{Lai-80}.

\subsubsection{Proof of Theorem \ref{thm:intertwining}}

The proof of theorem will be completed in two steps. First step is via explicit computations for $F$-rank one groups, and the second step is using the method of Bhanu-Murthy and Gindikin--Karpelevitch for reduction of higher rank case to that of rank one groups.

\paragraph{Case of rank one groups}

In the relative rank one groups there are four cases as described below. The expression for the intertwining operators should be understood to hold upto finitely many local factors which are holomorphic in the region $1-\epsilon < s < 1+\epsilon$ according to the Lemma \ref{lemma:holomorphy-local-intertwiner}. For certain meromorphic functions $f_1(s),\,f_2(s),\,f_3(s),$ and $f_4(s)$ of $s\in\mathbb{C}$, which are holomorphic near $s=1$, we have the following list of intertwining operators.

\begin{enumerate}
	\item $G = SL_2$,
    \[
        M(w_0, s\rho) = \frac{\zeta_F(s)}{\zeta_F(s+1)} = \zeta_F(s)f_1(s)
    \]
    \item $G = SU(3, E/F)$ where $E$ is a quadratic extension of $F$,
		\[
			M(w_0, s\rho) = \zeta_E(s) f_2(s)
		\]
	\item $G = \mathrm{Res}_{E'/F}(SL_2)$ then it follows from the proof of step 3 in Theorem \ref{thm:local-inter-rank-1-ur} that 
		\[
			M(w_0, s\rho) = \frac{\zeta_E(s)}{\zeta_E(s+1)} = \zeta_E(s)f_3(s)
		\]
  \item $G = \mathrm{Res}_{E'/F}SU(3, E/F)$ where $E$ is a quadratic extension of $E'$
		\[
			M(w_0, s\rho) = \zeta_{E'}(s) f_4(s).
		\]
\end{enumerate}
It is clear from the above list that the theorem holds for the $F$-rank one groups.


\paragraph{Case of higher rank groups}
We denote by $M^{G(\alpha)}(w_0,\lambda)$ the intertwining operator for the $F$-rank one semisimple simply connected group $G(\alpha)\subset G$ where $w_0$ is the largest element in the Weyl group of $G(\alpha)$. Writing $\lambda = (s_1, s_2, \dots, s_r)$ in the coordinate system given by $\xi$ (refer \eqref{eqn:coordinate-system}), Theorem \ref{thm:reduction} implies the following equality upto a scalar
\[
    M(w_0, \lambda) = \underset{\substack{\alpha_i\text{ positive}\\ \text{and simple}}}{\prod}M^{G(\alpha)}(w_0, \lambda|_{G(\alpha)}) \underset{\substack{\alpha\text{ positive, indi-}\\ \text{visible and nonsimple}}}{\prod}M^{G(\alpha)}(w_0, \lambda|_{G(\alpha)})
\]
For $s_i$ in the region $1-\epsilon < s_i < 1+\epsilon$, the poles of $M(w_0, (s_i))$ are determined by the poles of the operators on the right hand side.
In the case when $\alpha = \alpha_i$ is a positive simple root then $\lambda|_{G(\alpha_i)} = s_i$. If $\alpha$ is not a simple root then $\Re(\lambda|_{G(\alpha)})$ lies outside the domain $(1-\epsilon, 1+\epsilon)$.\\

\noindent
Note that $G(\alpha)$ is isomorphic to one of the four cases discussed above upto central isogeny. Reading the poles of the intertwining operators on the right hand side from the list for rank one groups, we can see that $M(w_0, \lambda)$ has simple poles along the hyperplanes $s_i=1$ when $(\Re(s_1), \Re(s_2), \dots, \Re(s_r)) \in (1-\epsilon, 1+\epsilon)^r$.\\

\noindent
The second part follows by restricting to the case of $\lambda = s\rho = (s, s, \dots, s)$.

\subsection{Prerequisites for the computation} \label{subsec:prerequistes-for-the-computation}

For any $h \in \mathcal{H}(G)$, the Hecke algebra, we define
\[
    T_h(\theta_{\varphi})(g) \colonequals\underset{A(F)\backslash A(\mathbbmss{A})}{\int} h(a^{-1}) \theta_{\varphi}(ag) \overline{da}.
\]   

\noindent
The operator $T_h$ enjoys the following property as can be seen from the integral representation above.
\begin{prop} \label{prop:operator-T-spectrum}
The operator $T_h$ defines a self-adjoint bounded operator on the closed subspace of $L^2(G(F)\backslash \mathbbmss{G}, \overline{\mu})$ generated by the function $\theta_{\varphi}$, 
such that if $\widehat{\psi}(\lambda) = \widehat{h}(\lambda)\widehat{\varphi}(\lambda)$, then $\theta_{\psi} = T_h(\theta_{\varphi})$.
The norm of $T_h$ is bounded above by $\widehat{h}(\rho)$.
\end{prop}
\begin{proof}
The existence of the operator $T_h$ follows from \cite[Lemma pp.136]{Morris-82}.
\end{proof}

\noindent
Let $\mathcal{E}^{\vee}$ be the closure of the subspace of $L^2(G(F)\backslash \mathbbmss{G})$ generated by the pseudo-Eisenstein series $\theta_{\varphi}$ where $\varphi$ is a compactly supported function on $A(\mathbbmss{A})/A(F)$.
Then the constant function belongs to $\mathcal{E}^{\vee}$ (See \cite[Ch. II, \S 1.12]{Moeglin-Waldspurger}). The main theorem of this section is the computation of the projection of the pseudo-Eisenstein series $\theta_{\varphi}$ onto the constant function.\\

\noindent
Choose $h \in \mathcal{H}(G)$ as below and consider the positive normal operator $T \colonequals T_h\circ (T_h)^*/(\widehat{h}(\rho))^2$.
\begin{enumerate}
	\item Choose a place $v_0 \notin S$ : via Satake isomorphism there exists $h_{v_0} \in \mathcal{H}(G\times_F F_{v_0})$ such that it's Fourier transform satisfies $\widehat{h}_{v_0}(s_{v_0}) = \sum_{w\in W_F} (\sharp k(v_0))^{-\langle \rho, ws_{v_0}\rangle}$.
	\item At places $v\neq v_0$ define $h_v$ to be the characteristic function of $K_v$.
\end{enumerate}
Following Harder, we prove:

\begin{theorem} \label{thm:projection-onto-constants}
The sequence of positive normal operators $T^n : \mathcal{E}^{\vee} \rightarrow \mathcal{E}^{\vee}$ converges to the operator $P : \mathcal{E}^{\vee} \rightarrow \mathcal{E}^{\vee}$ which is the projection onto the constant functions. Explicitly
	\[
		P(\theta_{\varphi}) = c\, \log(q)^r \mathrm{res}_{s=1}E(g, s\rho)\widehat{\varphi}(s\rho) = cc'\,\log(q)^r \lim_{s\to 1}(s-1)^rM(w_0, s\rho)\widehat{\varphi}(s\rho),
	\]
 where $c$ and $c'$ are the constants satisfying $d\lambda|_{\frac{X^*(A)\otimes \mathbb{R}}{2\pi /\log(q)X^*(A)}} = c \left(\frac{\log q}{2\pi}\right)^rdz_1 \wedge \dots \wedge dz_r$ and $\mathrm{res}_{\lambda=\rho}(E(x, \lambda)\widehat{\varphi}(\lambda)) =c'\lim_{s\to 1}(s-1)^rM(w_0, s\rho)\widehat{\varphi}(s\rho)$.
\end{theorem}

\begin{proof}
We have the following equality from the above theorem
\[
	T^n(\theta_{\varphi})(g) = \int_{\Lambda_{\sigma}(A)} E(g, \lambda)\widehat{h}(\lambda)^{2n}\widehat{h}(\rho)^{-2n}\widehat{\varphi}(\lambda) d\lambda.
\]
Note that the residue of the Eisenstein series $E(g, \widehat{\varphi}(\lambda))$ at $\lambda=\rho$ is a constant function. The proof henceforth is completely analogous to the proof given in \cite[pp. 301, 303]{Harder-74}.
We summarize the main steps below.
In the equations below $\sigma'$ is a real quasi-character such that $\sigma'_{i} < 1$ for some $i$ where $\sigma'_i$ are the coordinates of $\sigma'$ in the coordinate system given by $\xi$, and $\widetilde{E}(g, \lambda)$ denotes the Eisenstein series or it's residue.
\begin{align*}
T^n(\theta_{\varphi})(g) &= c\; \log(q)^r \; \mathrm{res}_{\lambda=\rho}(E(g, \lambda)\widehat{\varphi}(\lambda)) + \sum T^n\left( \underset{\Lambda_{\sigma'}(A)}{\int} \widetilde{E}(g, \lambda)\widehat{\varphi}(\lambda)d\lambda \right)\\
                      &= c\; \log(q)^r \; \mathrm{res}_{\lambda=\rho}(E(g, \lambda)\widehat{\varphi}(\lambda)) + \sum \left( \underset{\Lambda_{\sigma'}(A)}{\int} \widetilde{E}(g, \lambda)\widehat{\varphi}(\lambda) \left(\frac{\widehat{h}(\lambda)}{\widehat{h}(\rho)}\right)^{2n}d\lambda \right)
\end{align*}
Note that for $\lambda \in \Lambda_{\sigma'}(A)$, the inequality $\widehat{h}(\lambda) < \widehat{h}(\rho)$ holds (See Lemma \ref{lemma:maximal-value-vertices}). Hence we get
\begin{equation} \label{eqn:projection}
    \lim_{n \to \infty} T^n(\theta_{\varphi}) = c\; \log(q)^r \; \mathrm{res}_{\lambda=\rho}(E(x, \lambda)\widehat{\varphi}(\lambda)).
\end{equation}
The above limit and the equality is to be understood as pointwise convergence. Proposition \ref{prop:operator-T-spectrum} implies that the spectrum of the self-adjoint positive operator $T$ is concentrated on $[0,1]$ and hence $T^n \to P$ where $P$ is the projection onto the subspace
\[
    \lbrace e \in \mathcal{E}^{\vee} \suchthat Te = e \rbrace.
\]
This observation of Harder coupled with the pointwise convergence result from equality \eqref{eqn:projection} implies that the equality \eqref{eqn:projection} in fact holds in $L^2(G(\mathbb{Q})\backslash G(\mathbbmss{A}))$. This finishes the proof of the first equality.
Following the arguments in \cite[pp.289, 290]{Harder-74} we get that $\mathrm{res}_{s=1}E(g, s\rho)\widehat{\varphi}(s\rho) = q^{\dim(N)(1-g)}\mathrm{res}_{s=1}E^{B}(g, s\rho)\widehat{\varphi}(s\rho)$.
Now using the formula for the constant term from Lemma \ref{lemma:morris} and observing that the intertwining operators $M(w, s\rho)$ has poles of order $< r$ for $w\neq w_0$ we get the second equality with $c' = q^{\dim(N)(1-g)}$. 
\end{proof}

\subsection{A final computation} \label{subsec:final-computation}
We will complete the proof of the Weil conjecture in the case of quasi-split group over function fields in this section.\\

\noindent
We begin with the equality
\begin{align*}
P\theta_{\varphi} = & \, cc' \log(q)^r \; \lim_{t\to 1}(t-1)^rM(w_0,t\rho)\widehat{\varphi}(t\rho)\\
				   = & \, cc'\; \log(q)^r\;\mathrm{res}_{s=1}(L^S(s,X^*(A)))\lim_{t\to 1} \frac{M(w_0,t\rho)\widehat{\varphi}(t\rho)}{L(t,A)}\\
				     = & \, cc'\; \log(q)^r \; \mathrm{res}_{s=1}(L^S(s,X^*(A))) \lim_{t\to 1} \frac{M^S(w_0,t\rho)\widehat{\varphi^S}(t\rho)}{L^S(t, X^*(A))}\prod_{v\in S} M_v(w_0,\rho)\widehat{\varphi_v}(t\rho)\\
				     = & \, cc'\; \log(q)^r \; \mathrm{res}_{s=1}(L^S(s,X^*(A)))\lim_{t\to 1} \frac{M^S(w_0,t\rho)\widehat{\varphi^S}(t\rho)}{L^S(t, X^*(A))}\prod_{v\in S} M_v(w_0,\rho)\widehat{\varphi_v}(t\rho)\\
				     = & \, cc'\; \log(q)^r\; \mathrm{res}_{s=1}(L^S(s,X^*(A))) \left(\prod_{v\notin S}\mathrm{vol}(K_v)\widehat{\varphi^S}(t\rho)\right) \left(\kappa \left(\prod_{v\notin S}\mathrm{vol}(K_v)\right)^{-1} \prod_{v\in S}\widehat{\varphi_v}(t\rho) \right)\\
				     = & \, cc'\; \log(q)^r\; \mathrm{res}_{s=1}(L^S(s,X^*(A))) \kappa\widehat{\varphi}(\rho).
\end{align*}

\noindent
Since $P$ is the projection operator onto the constants we have the equality $(\theta_{\varphi}, 1) = (P\theta_{\varphi}, 1)$. The right hand side equals $q^{-\dim(G)(1-g)}\tau(G)cc'\, \log(q)^r \mathrm{res}_{s=1}(L(s,X^*(A))) \kappa\widehat{\varphi}(\rho)$.
Since we can surely have functions $\varphi$ with $\widehat{\varphi}(\rho) \neq 0$, we get the equality
\[
	\tau(G) = \frac{q^{(\dim(G)-\dim(N))(1-g)}}{cc'\;\log(q)^r \mathrm{res}_{s=1}(L(s,X^*(A)))} = \tau(A).
\]
The last equality follows from the explicit value of $c$ obtained in \S \ref{subsubsec:constant-c} and of $c'$ obtained in the proof of Theorem \ref{thm:projection-onto-constants}. We know from \cite[Ch.II, Theorem 1.3(d)]{Oesterle-84} that $\tau(\mathrm{Res}_{E/F}(\mathbb{G}_m)) = \tau(\mathbb{G}_m)$ and hence $\tau(A) = 1$.
Using Lemma \S \ref{section:appendix-quasi-split-tori} and the fact that the Tamagwa number of split tori is 1 we get that
\[
    \tau(G) = \tau(A) = 1.
\]

\appendix
\section{Dual groups and restriction of scalars}
\noindent
Let $E\supset E' \supset F$ be a tower of unramified extensions of local fields. Let $A'$ be a torus defined over $E'$ which splits over $E$ and consider $A = \mathrm{Res}_{E'/F}A'$.
We have the $\mathrm{Gal}(E/E')$-equivariant isomorphism
\begin{equation}\label{eqn:gal_equi_isomorphism}
\widehat{A}^{\mathrm{Gal}(E/E')} \cong \prod_{\mathrm{Gal}(E'/F)}\widehat{A'}^{\mathrm{Gal}(E/E')}
\end{equation}
and the action of $\mathrm{Gal}(E'/F)$ is given by permuting the indices. Hence
\[
    \widehat{A}^{\mathrm{Gal}(E/F)} \cong \widehat{A'}^{\mathrm{Gal}(E/E')}.
\]
The inclusion $\widehat{A'}^{\mathrm{Gal}(E/E')} \hookrightarrow \widehat{A}^{\mathrm{Gal}(E/E')}$ can be identified under the isomorphism \eqref{eqn:gal_equi_isomorphism} with the diagonal embedding $\widehat{A'}^{\mathrm{Gal}(E/E')} \hookrightarrow \prod_{\mathrm{Gal}(E'/F)}\widehat{A'}^{\mathrm{Gal}(E/E')}$.
Define the map $\eta : \widehat{A}(\mathbb{C}) \rightarrow X^*(A\times \bar{F}) \otimes \mathbb{C}$ by the condition that $\mu(\widehat{t}) = |\pi_F|^{\langle \eta(\widehat{t}), \mu\rangle}$ for all $\mu \in X_*(A\times \bar{F})$.
Similarly, we may define $\eta' : \widehat{A'}(\mathbb{C}) \rightarrow X^*(A'\times \bar{F}) \otimes \mathbb{C}$.
\begin{lemma} \label{lemma:dual-tori-eta-map}
    Let $\widehat{t} \in \widehat{A}^{\mathrm{Gal}(E'/F)}$, $\widehat{t'} \in \widehat{A'}^{\mathrm{Gal}(E/E')}$ be such that under the isomorphism \eqref{eqn:gal_equi_isomorphism} we have $\widehat{t} = (\widehat{t'}, \widehat{t'}, \dots, \widehat{t'})$.
    Further assume that $\lambda = \eta(\widehat{t})$ and $\mathrm{Nm}_{E'/F}(\lambda') = \lambda$, then $\eta'(\widehat{t'}^n) = \lambda'$.
\end{lemma}
\begin{proof}
Note that there is the following commutative diagram
\begin{equation} \label{eqn:gal-eq-pairing}
\begin{tikzcd}[row sep = large]
    X_*(A') \times X^{*}(A') \arrow[d, xshift=5ex, "\mathrm{Nm}_{E'/F}"] \arrow[r] &\mathbb{Z} \arrow[d, "\times \text{[}E':F\text{]}"]\\
    X_*(A) \times X^*(A) \arrow[u, xshift=-5ex] \arrow[r] & \mathbb{Z},
\end{tikzcd}
\end{equation}
where the left most vertical arrow is an isomorphism given by the adjunction of restriction and extension of scalars. For $\widehat{t},\,\lambda,\,\lambda'$ as in the statement of the lemma, and $\mu \in X_*(A)$,
\begin{align*}
    \mu(\widehat{t}) &= |\pi_F|^{\langle \eta(\widehat{t}),\mu \rangle} = |\pi_F|^{\langle \mathrm{Nm}_{E'/F}(\lambda'),\mu \rangle} = |\pi_F|^{n\langle \lambda',\mu \rangle} = |\pi_{E'}|^{\langle \lambda',\mu\rangle}.
\end{align*}
Recall that $\widehat{t} = (\widehat{t'}, \widehat{t'}, \dots, \widehat{t'})$, hence $\mu(\widehat{t}) = \mu(\widehat{t'})^n$ for any $\mu \in X_*(A) = X_*(A')$.
Thus, we get $|\pi_{E'}|^{\langle \lambda',\mu\rangle} = \mu(\widehat{t'}^n)$. Hence by definition of $\eta'$ we get $\eta'(\widehat{t'}^n) = \lambda'$.
\end{proof}

\section{Quasi-split tori in simply connected groups} \label{section:appendix-quasi-split-tori}

\noindent
We state the following lemma from \cite[Lemma 6.1.2]{Rapoport-tamagawa} for the sake of completeness

\begin{lemma} \label{lemma:quasisplit-tori-structure}
Suppose $G$ is a simply connected quasi-split group over a field $F$. Let $A$ be a maximal torus defined over $F$ which is contained in a Borel subgroup defined over $F$.
Then $A$ is a product of tori of the form $\mathrm{Res}_{E_i/F}\mathbb{G}_m$, where $E_i/F$ are finite separable extension of $F$.
\end{lemma}

\begin{proof}
Let $X^*(A\times F^{sep})$ be the set of characters of $A$ defined over $F^{sep}$. Then the Galois group $\mathrm{Gal}(F^{sep}/F)$ acts on the group $X^*(A\times F^{sep})$.
When $G$ is quasi-split the restriction map $\Pi_{F^{sep}} \rightarrow \Pi_F$ is surjective and the fibers are exactly the $\mathrm{Gal}(F^{sep}/F)$-orbits. This implies that the set of absolute simple roots restricting to a given relative simple root is permuted by the Galois group $\mathrm{Gal}(F^{sep}/F)$.
We may use \cite[Exercise 13.1.5(4)]{Springer98} to conclude the lemma.
\end{proof}

\section{A lemma}

\begin{lemma} \label{lemma:maximal-value-vertices}
The inequality $\widehat{h}(\lambda) < \widehat{h}(\rho)$ holds.
\end{lemma}
\begin{proof}
The proof follows as in \cite[ Lemma 3.2.3]{Harder-74} which in turn depends on \cite[Lemma 3.2.1(ii)]{Harder-74}. We need only prove an analogous result to the latter Lemma for the quasi-split case.
That is, to show that
\[
    \lbrace \sigma = (\sigma_i) \suchthat 1-\epsilon < \Re(\sigma_i) \leq 1 \; \forall \; i \rbrace \subset \lbrace \sigma \suchthat \Re(\sigma) \in \mathrm{ConvHull}(W_F\cdot \rho) \rbrace.
\]
Note that the restriction map $X^*(A \times \bar{F}) \twoheadrightarrow X^*(A)$ in our chosen coordinate system \eqref{eqn:coordinate-system} can be identified with the map `average over the Galois orbits'. This is a convex map and hence preserves convex domains.
Since the lemma is known for the convex hull of the Weyl conjugate of $\rho$ in $X^*(A \times \bar{F})$, the lemma follows in the quasi-split case as well.
\end{proof}

\end{document}